\documentclass[10,5pt,a4paper]{article}
\usepackage{etex} 
\usepackage[latin1]{inputenc}
\usepackage[T1]{fontenc}
\usepackage[right=3cm,left=3cm,top=1cm,bottom=1cm]{geometry}
\usepackage{pstricks}
\usepackage{enumerate}
\usepackage{amsmath}
\usepackage{nccmath} 
\usepackage{amsthm}
\usepackage{amsfonts}
\usepackage{amssymb} 
\usepackage[np]{numprint}
\usepackage{fourier}

\npdecimalsign{\ensuremath{.}}

\title{A note on the $q$-adic valuation of $\sigma_k(n)$}
\author{Olivier Bordell\`{e}s}
\date{}

\newcommand{\Z}{\mathbb {Z}}

\newcommand{\sdfrac}[2]{\mbox{\small$\displaystyle\frac{#1}{#2}$}}

\let\oldqedhere\qedhere
\renewcommand{\qedhere}{\pushQED{\qed}\oldqedhere}

\DeclareMathOperator{\md}{mod}
\DeclareMathOperator{\ord}{ord}

\theoremstyle{theorem}
\newtheorem{theorem}{Theorem}
\newtheorem{lem}[theorem]{Lemma}
\newtheorem{coro}[theorem]{Corollary}

\theoremstyle{definition}
\newtheorem{exe}[theorem]{Example}

\theoremstyle{remark}

\providecommand{\keywordsubject}[1]{\textbf{2020 Mathematics Subject Classification :} \: #1}
\providecommand{\keywords}[1]{\textbf{Keywords :} #1}

\begin{document}

\maketitle

\footnote{\keywordsubject{11A25, 11A41, 11B50, 11B83}}

\footnote{\keywords{Sum of divisors, valuation, cyclotomic polynomials}}

\begin{abstract} 
In this note, we obtain an exact formula for the $q$-adic valuation of $\sigma_k(n)$ where $q$ is an odd prime, allowing us to derive an explicit upper bound which is asymptotically better than the previous bound obtained by Zhao when $n$ is large and $k \geqslant q-2$. The key parts are played by the LTE lemma and the use of cyclotomic polynomials.
\end{abstract}

\section{Introduction}

\subsection{Main results}

\noindent
Let $q$ be a prime number and $k \in \Z_{\geqslant 1}$. In recent years, several papers have appeared dealing with the $q$-adic valuation of $\sigma_k(n) := \sum_{d \mid n} d^{k}$. For instance, when $k=1$ and $q=2$, the authors in \cite{amd21} proved that $v_2(\sigma(n)) \leqslant \left \lceil \frac{\log n}{\log 2} \right \rceil$, and also obtained conditional upper bounds for $v_q(\sigma(n))$ when $q$ is odd. These conditions were next dropped off in \cite{zhao25} where the authors showed that $v_q(\sigma(n)) \leqslant \left \lceil \frac{\log n}{\log q} \right \rceil $ whenever $q \geqslant 3$. Subsequently, Zhao \cite{zhao26} generalized this result to the function $\sigma_k(n)$, establishing the bound
\begin{equation}
   v_q \left( \sigma_k(n) \right) \leqslant \left \lceil \frac{k \log n}{\log q} \right \rceil \label{eq:zhao}
\end{equation}
for all $n \geqslant 2$, $k \geqslant 1$ and primes $q \geqslant 2$. The case $q=2$ was later improved by Cheng and Zhang in \cite{cheng26} in which it is proven that
\begin{equation}
   v_2 \left( \sigma_k (n) \right) = \sum_{\substack{p^{\alpha} \| n \\ p \ne 2 \\ 2 \mid \alpha + 1}} \biggl( v_2(\alpha+1) + v_2 \left( p^k + 1 \right) - 1 \biggr) \label{eq:cheng}
\end{equation}
and from which the authors derived the bound
$$v_2 \left( \sigma_k(n) \right) \leqslant \begin{cases} \left \lfloor \frac{\log n}{\log q} \right \rfloor, & \mathrm{if} \ 2 \mid k \, ; \\ & \\ \left \lceil \frac{\log n}{\log q} \right \rceil, & \mathrm{if} \ 2 \nmid k. \end{cases}$$
The aim of this note is to generalize \eqref{eq:cheng} to every prime $q \geqslant 3$, and then derive an upper bound for $v_q \left( \sigma_k(n) \right)$ which is asymptotically better than \eqref{eq:zhao} when $n$ is large and $k \geqslant q-2$.

\begin{theorem}
\label{th:valuation_sigma_k}
Let $n \in \Z_{\geqslant 2}$, $k \in \Z_{\geqslant 1}$ and $q \geqslant 3$ be a prime number. Then
$$v_q \left( \sigma_k(n)\right) =\sum_{\substack{p^{\alpha} \| n \\ p^{k} \equiv 1 \; (\md q)}} v_q(\alpha+1) + \sum_{\substack{p^{\alpha} \| n \\ p \ne q \\ p^k \not \equiv 1 \; (\md q) \\ \varpi_k \mid \alpha + 1}} \Bigl( v_q(\alpha+1) + v_q(k) + v_q \left( \Phi_{\varpi_1} (p) \right) \Bigr)$$
where $\varpi_1 := \ord_q (p)$, $\varpi_k := \ord_q \bigl(p^k \bigr)$ and $\Phi_m(X)$ is the $m$th cyclotomic polynomial.
\end{theorem}

\begin{exe}
\label{ex:valuation_sigma_k}
With $q=5$ and $k=2$, Theorem~\ref{th:valuation_sigma_k} yields
$$v_5 \left( \sigma_2(n)\right) = \sum_{\substack{p^{\alpha} \| n \\ p \equiv \pm 1 \; (\md 5)}} v_5(\alpha+1) + \sum_{\substack{p^{\alpha} \| n \\ p \equiv \pm 2 \; (\md 5) \\ 2 \mid \alpha + 1}} \Bigl( v_5(\alpha+1) + v_5 \left( p^{2} + 1 \right) \Bigr) .$$
With $q=7$ and $k=9$, we derive
\begin{align*}
   v_7 \left( \sigma_9(n)\right) &= \sum_{\substack{p^{\alpha} \| n \\ p \equiv 1,2,4 \; (\md 7)}} v_7(\alpha+1) + \sum_{\substack{p^{\alpha} \| n \\ p \equiv 3,5 \; (\md 7) \\ 2 \mid \alpha+1}} \Bigl( v_7(\alpha+1) + v_7(p^2-p+1) \Bigr) \\
   & \hspace*{0.5cm} + \sum_{\substack{p^{\alpha} \| n \\ p \equiv 6 \; (\md 7) \\ 2 \mid \alpha+1}} \Bigl( v_7(\alpha+1) + v_7(p+1) \Bigr).
\end{align*}
\end{exe}

\begin{coro}
\label{cor:valuation_sigma_k}
Let $n \in \Z_{\geqslant 3}$, $k \in \Z_{\geqslant 1}$ and $q \geqslant 3$ be a prime number. Then
$$v_q \left( \sigma_k(n)\right) < \frac{M_{q,k} \log \gamma(n)}{\log q} + \frac{(\np{1.385} \log k + \np{1.066})\log n}{\log q \log \log n}+ \frac{4(\log \log \log n + 1)}{\log q}.$$
Here, $M_{q,k} := \max \varphi(d)$ where the maximum runs over the divisors $ d$ of $q-1 $ such that $d \nmid k$, with the convention $M_{q,k} = 0$ if this set is empty. In particular
$$v_q \left( \sigma_k(n)\right) < \frac{(q-2)\log \gamma(n)}{\log q}  + \frac{(\np{1.385} \log k + \np{1.066})\log n}{\log q \log \log n}+ \frac{4(\log \log \log n + 1)}{\log q}.$$
Furthermore, if $v_q(k)=0$, then the term $\np{1.385} \log k$ can be omitted.
\end{coro}

\noindent
Note that $M_{q,k}$ may be much smaller than $\min(q-2,k)$ since the use of the cyclotomic polynomial $\Phi_{\varpi_1}(X)$ may reduce drastically the $q$-adic valuation of the second sum in Theorem~\ref{th:valuation_sigma_k}. For instance, $M_{7,9} = 2$ as can be seen in Example~\ref{ex:valuation_sigma_k}. Also note that Theorem~\ref{th:valuation_sigma_k} implies that, if $p \ne q$, $p^k \not \equiv 1 \; (\md q)$ and $\ord_q \left( p^{k}\right) \mid \alpha + 1$, then $ v_q \left( \sigma_k \left( p^{\alpha}\right) \right)  \geqslant 1 $.

\subsection{Notation}

\noindent
In what follows, $p$ and $q$ are always prime numbers satisfying $p \ne q$ and $q \geqslant 3$. The notation $p^{\alpha} \| n$ means $p^{\alpha} \mid n$ and $p^{\alpha + 1} \nmid n$, so that $v_p(n) = \alpha$. The $\gcd$ of the integers $a$ and $b$ is denoted by $(a,b)$. For all $k \in \Z_{\geqslant 1}$, $\varpi_k := \ord_q \bigl(p^k \bigr)$ and $d_k := (k,\varpi_1) = \Bigl( k, \ord_q(p)\Bigr)$. If $n \geqslant 2$, $\varphi(n)$ is the Euler totient, $\tau(n) := \sum_{d \mid n} 1$ is the number of divisors of $n$, $\omega(n) := \sum_{p \mid n} 1$ is the number of distinct prime divisors of $n$ and $\gamma(n) := \prod_{p \mid n} p$ is the squarefree kernel of $n$. Finally, $\Phi_m(X)$ is the $m$th cyclotomic polynomial.

\section{Proofs}

\subsection{The LTE Lemma}

\noindent
The next result, which seems to appear for the first time in 1878 in the work of Lucas \cite[Section~XIII]{luc78}, will play a key part in the proof of Theorem~\ref{th:valuation_sigma_k}. For a proof, see for instance \cite[Theorem~1]{manea06}.

\begin{lem}[LTE]
\label{le:LTE}
Let $a \ne b \in \Z$, $q \geqslant 3$ prime such that $q \nmid ab$ and $q \mid a-b$, and $m \in \Z_{\geqslant 1}$. Then
   $$v_q \left( a^{m} - b^{m} \right) = v_q(a-b) + v_q(m).$$
\end{lem}

\subsection{Tools for values of cyclotomic polynomials}

\noindent
The following lemma is quite well-known and may follow from the fact that $\Phi_m(X)$ is a reciprocal polynomial for instance. We provide here a slightly different proof.

\begin{lem}
\label{le:bounds_cyclo}
Let $n \in \Z_{\geqslant 2}$ and $m \in \Z_{\geqslant 1}$. Then
$$\prod_{k=1}^{\infty} \left( 1 - \sdfrac{1}{n^k}\right) n^{\varphi(m)} \leqslant \Phi_m(n) \leqslant \prod_{k=1}^{\infty} \left( 1 - \sdfrac{1}{n^k}\right)^{-1} n^{\varphi(m)}.$$ 
\end{lem}

\begin{proof} 
By M\"{o}bius inversion
$$\Phi_m(n) = \prod_{k \mid m} \left(n^{k} - 1 \right)^{\mu(m/k)} = \prod_{k \mid m} n^{k\mu(m/k)} \ \prod_{k \mid m}\left(1 - \sdfrac{1}{n^k} \right)^{\mu(m/k)}$$
and
$$\log \left( \prod_{k \mid m} n^{k\mu(m/k)}\right) = \log n \sum_{k \mid m} k \mu \left( \frac{m}{k} \right)  = \varphi(m) \log n$$
where we used the Dirichlet convolution identity $\mathrm{id} \star \mu = \varphi$, so that
$$\Phi_m(n) = n^{\varphi(m)} \prod_{k \mid m} \left( 1 - \sdfrac{1}{n^k}\right)^{\mu(m/k)}$$
and the asserted bounds follow easily.
\end{proof}

\noindent
The next result is well-known. See \cite[Lemma 2.9]{was82} for instance.

\begin{lem}
\label{le:cyclo_val}
Let $q$ be a prime. Then for all $n \in \Z$ and $m \in \Z_{\geqslant 1}$ such that $q \nmid mn$, we have
$$q \mid \Phi_m(n) \iff m = \ord_q(n).$$
\end{lem}

\subsection{Proof of Theorem~\ref{th:valuation_sigma_k}}

\noindent
Write $n = q^{\nu} a b$ with $\nu \in \Z_{\geqslant 0}$, $(a,b)=1$ such that $p \mid a \Longrightarrow p^{k} \equiv 1 \; (\md q)$ and $p \mid b \Longrightarrow \left( p \ne q \right.$ and $\left. p^{k} \not \equiv 1 \; (\md q) \right)$. Note that $\sigma_k \bigl( q^{\nu} \bigr) = 1 + q^k  + \dotsb + q^{k \nu} \equiv 1 \; (\md q)$, and hence $v_q \left( \sigma_k \bigl( q^{\nu} \bigr) \right) =0$. Thus
\begin{equation}
   v_q \left( \sigma_k(n)\right) = v_q \left( \sigma_k(a)\right) + v_q \left( \sigma_k(b)\right). \label{eq:sigma_k}
\end{equation}
If $q \geqslant 3$, $q \ne p$ and $q \mid p^k - 1$, we apply Lemma~\ref{le:LTE} with $a=p^k$, $b=1$ et $m= \alpha +1$, yielding
$$v_q \left( \sigma_k \left( p^{\alpha}\right) \right) = v_q \left( \frac{p^{k(\alpha+1)}-1}{p^k-1} \right)= v_q \left( p^{k(\alpha+1)} - 1 \right) - v_q \left( p^{k} - 1 \right)= v_q(\alpha + 1)$$
so that
\begin{equation}
   v_q \left( \sigma_k(a)\right) = \sum_{p^{\alpha} \| a} v_q(\alpha+1) =\sum_{\substack{p^{\alpha} \| n \\ p^{k} \equiv 1 \; (\md q)}} v_q(\alpha+1). \label{eq:sigma_a}
\end{equation}
To evaluate $v_q \left( \sigma_k(b)\right)$, we assume throughout that $p^k \not \equiv 1 \; (\md q)$ and first note that, for all $p^{\alpha} \| b$, we have
$$v_q \left( \sigma_k \bigl( p^{\alpha} \bigr) \right) = v_q \left( p^{k(\alpha+1)}-1 \right) $$
since $q \nmid p^{k}-1$. Also notice that this latter condition implies $\varpi_1 \geqslant 2$ and $\varpi_k \geqslant 2$.
\begin{enumerate}
   \item[\scriptsize $\triangleright$] If $ \varpi_k \nmid \alpha + 1 $, the euclidean division of $\alpha + 1$ by $\varpi_k$ yields $\alpha + 1 = h \varpi_k + r$ with $1 \leqslant r < \varpi_k$, so that
   $$p^{k(\alpha+1)}-1 = \left( p ^{k \varpi_k} \right)^{h} \times p^{rk} - 1 \equiv p^{rk} - 1 \not \equiv 0 \; (\md q)$$
   otherwise we would have $\varpi_k \leqslant r$ since $r \geqslant 1$, which would contradict the inequality $r < \varpi_k$. We infer that, in this case, we have
   $$v_q \left( \sigma_k\bigl( p^{\alpha} \bigr)\right) = 0.$$
   \item[\scriptsize $\triangleright$] If $ \varpi_k \mid \alpha + 1 $, write $\alpha + 1 = h \varpi_k$ with $h \in \Z_{\geqslant 1}$. Since $p^{k \varpi_k} \equiv 1 \; (\md q)$, Lemma~\ref{le:LTE} yields
   \begin{align*}
      v_q \left( \sigma_k\bigl( p^{\alpha} \bigr)\right) &= v_q \left( p^{k h \varpi_k} - 1 \right) = v_q \left( p^{k  \varpi_k} - 1 \right) + v_q(h) \\
      &= v_q \left( p^{k \varpi_k} - 1 \right) + v_q(\alpha+1) - v_q(\varpi_k) \\
      &= v_q \left( p^{k \varpi_k} - 1 \right) + v_q(\alpha+1)
   \end{align*}
   where we used $\varpi_k \mid q-1$ and hence $ v_q(\varpi_k) = 0 $. Therefore
   \begin{align*}
      v_q \left( \sigma_k(b)\right) &= \sum_{\substack{p^{\alpha} \| b \\ \varpi_k \mid \alpha + 1}} \left( v_q(\alpha+1) + v_q \left( p^{k \varpi_k} - 1 \right) \right) \\
      & = \sum_{\substack{p^{\alpha} \| n \\ p \ne  q \\ p^{k} \not \equiv 1 \; (\md q) \\ \varpi_k  \mid \alpha + 1}} \left( v_q(\alpha+1) + v_q \left( p^{k \varpi_k} - 1 \right) \right).
   \end{align*}
   We now distinguish between two cases:
   \begin{enumerate}[(i)]
      \item If $q \mid p^{\varpi_k} - 1$, Lemma~\ref{le:LTE} implies that
      $$v_q \left( p^{k \varpi_k} - 1 \right) = v_q \left( p^{\varpi_k} - 1 \right) + v_q(k).$$
      Also note that, since $1 \equiv p^{k \varpi_k} \equiv p^{\varpi_k} \; (\md q)$, we derive $\varpi_1 \mid \varpi_k (k-1)$, and since  $\varpi_1 = \varpi_k d_k $ where $d_k := \left( \varpi_1,k\right)$, we get $d_k \mid k-1$. Since $d_k \mid k$, we obtain $d_k = 1$ and then $\varpi_k = \varpi_1$. Hence
      $$v_q \left( p^{k \varpi_k} - 1 \right) = v_q \left( p^{\varpi_1} - 1 \right) + v_q(k).$$
      \item Assume $q \nmid p^{\varpi_k} - 1$. Using again $\varpi_1 = \varpi_k d_k $, we derive $ k \varpi_k =  \frac{k \varpi_1}{d_k} := m \varpi_1$, where $m:= \frac{k}{d_k}$. Since $q \mid p^{\varpi_1} - 1$, Lemma~\ref{le:LTE} yields
      \begin{align*}
         v_q \left( p^{k \varpi_k} - 1 \right) &= v_q \left(p^{m \varpi_1}  - 1\right)  \\
         &= v_q \left( p^{\varpi_1} - 1 \right) + v_q(m) \\
         &= v_q \left( p^{\varpi_1} - 1 \right) + v_q (k) - v_q (d_k).
      \end{align*}
   Furthermore, since $\varpi_1 \mid q-1$, we have $q \nmid \varpi_1$, so that $ v_q (d_k) = 0 $.
   \end{enumerate}
   Hence
   $$v_q \left( \sigma_k(b)\right) = \sum_{\substack{p^{\alpha} \| n \\ p \ne  q \\ p^{k} \not \equiv 1 \; (\md q) \\ \varpi_k  \mid \alpha + 1}} \left( v_q(\alpha+1) + v_q(k) +v_q \left( p^{\varpi_1} - 1 \right) \right).$$
   Now
   $$v_q \left( p^{\varpi_1} - 1 \right) = v_q \left( \prod_{d \mid \varpi_1} \Phi_d (p) \right) = \sum_{d \mid \varpi_1} v_q \left( \Phi_d (p) \right)$$
   and since $q \ne p$ and $q \nmid d$ for all $d \mid \varpi_1$ because $\varpi_1 \mid q-1$, Lemma~\ref{le:cyclo_val} yields $v_q \left( \Phi_d (p) \right) \ne 0 \iff d= \ord_q(p)= \varpi_1$, so that
   $$v_q \left( p^{\varpi_1} - 1 \right) = v_q \left( \Phi_{\varpi_1} (p) \right).$$
\end{enumerate}
   Finally, we get
   \begin{equation}
      v_q \left( \sigma_k(b)\right) = \sum_{\substack{p^{\alpha} \| n \\ p \ne q \\ p^k \not \equiv 1 \; (\md q) \\ \varpi_k \mid \alpha + 1}} \left( v_q(\alpha+1) + v_q(k) + v_q \left( \Phi_{\varpi_1} (p) \right)\right) \label{eq:sigma_b}
   \end{equation}     
   and we complete the proof by reporting \eqref{eq:sigma_a} and \eqref{eq:sigma_b} in \eqref{eq:sigma_k}. \qedhere
   
\subsection{Proof of Corollary~\ref{cor:valuation_sigma_k}}

\noindent
First note that, when $p^k \equiv 1 \, (\md q)$, then $\varpi_k = 1$, so that Theorem~\ref{th:valuation_sigma_k} may be restated as
$$v_q \left( \sigma_k(n)\right) = \sum_{\substack{p^{\alpha} \| n \\ p \ne q \\ \varpi_k \mid \alpha + 1}} v_q(\alpha+1)  + \sum_{\substack{p^{\alpha} \| n \\ p \ne q \\ p^k \not \equiv 1 \; (\md q) \\ \varpi_k \mid \alpha + 1}} \Bigl( v_q(k) +  v_q \left( \Phi_{\varpi_1} (p) \right) \Bigr)$$
and therefore
\begin{align}
   v_q \left( \sigma_k(n)\right) & \leqslant \sum_{\substack{p^{\alpha} \| n \\ p \ne q \\ \varpi_k \mid \alpha + 1}} \left( v_q(\alpha+1) + v_q(k) \right) + \sum_{\substack{p^{\alpha} \| n \\ p \ne q \\ p^k \not \equiv 1 \; (\md q) \\ \varpi_k \mid \alpha + 1}} v_q \left( \Phi_{\varpi_1} (p) \right) \label{eq:th_1_new} \\
   & := S_1+S_2, \notag
\end{align}
say.
\begin{enumerate}
   \item[\scriptsize $\triangleright$] The sum $S_1$ does not exceed
$$S_1 \leqslant \frac{1}{\log q} \left( \sum_{p^{\alpha} \| n} \log (\alpha + 1) + \log k \sum_{p \mid n} 1 \right) = \frac{\log \tau(n) + \omega(n) \log k}{\log q}$$
and the bounds \cite[Th\'{e}or\`{e}me~1]{nic83} and \cite[Th\'{e}or\`{e}me~11]{rob83} imply that
\begin{equation}
   S_1 < \frac{\left( \np{1.385} \log k + \np{1.066} \right) \log n}{\log q \log \log n}. \label{eq:maj_sum_1}
\end{equation}
   \item[\scriptsize $\triangleright$] It remains to estimate $S_2$. Let us first note that the condition $p^k \not \equiv 1 \; (\md q)$ is equivalent to $\varpi_1 \mid q-1$ and $\varpi_1 \nmid k$. Also recall that $M_{q,k} = \max \varphi(d)$ where the maximum is taken over the divisors $ d$ of $q-1 $ such that $d \nmid k$, with the convention $M_{q,k} = 0$ if this set is empty. Hence, in view of the previous remark, we get the bound $\varphi (\varpi_1) \leqslant M_{q,k}$ for all primes $p \mid n$ satisfying $p^k \not \equiv 1 \; (\md q)$.
   
   Using Lemma~\ref{le:bounds_cyclo}, we derive for $n \geqslant 3$
   \begin{align}
      S_2 & \leqslant \frac{1}{\log q} \sum_{\substack{p \mid n \\ p^k \not \equiv 1 \; (\md q)}} \log \left( \Phi_{\varpi_1} (p) \right) \notag \\
      & \leqslant \frac{1}{\log q} \sum_{\substack{p \mid n \\ p^k \not \equiv 1 \; (\md q)}} \Biggl( \log \left( p^{\varphi(\varpi_1)}\right) - \sum_{k=1}^{\infty} \log \left( 1 - \frac{1}{p^k}\right) \Biggr) \notag \\
      & \leqslant \frac{1}{\log q} \log \left( \prod_{\substack{p \mid n \\ p^k \not \equiv 1 \; (\md q)}} p^{\varphi(\varpi_1)} \right) + \frac{1}{\log q} \sum_{p \mid n} \sum_{k=1}^{\infty} \left( \frac{1}{p^k} + \frac{1}{p^{2k}}\right) \notag \\
      & \leqslant \frac{1}{\log q} \log \left( \prod_{\substack{p \mid n \\ p^k \not \equiv 1 \; (\md q)}} p^{M_{q,k}} \right) + \frac{1}{\log q} \sum_{p \mid n} \left( \frac{1}{p-1} + \frac{1}{p^2-1} \right)  \notag \\
      &< \frac{M_{q,k}}{\log q} \log \left( \prod_{p \mid n} p \right) + \frac{1}{\log q} \sum_{p \mid n} \left( \frac{1}{p} + \frac{4}{p^2} \right) \notag \\
      & \leqslant \frac{M_{q,k} \log \gamma(n)}{\log q} + \frac{1}{\log q} \left( \log \log \log n + 2 + 4 \log \zeta(2) \right) \notag \\
      & < \frac{M_{q,k} \log \gamma(n)}{\log q}  + \frac{4(\log \log \log n + 1)}{\log q} \label{eq:maj_sum_2}
   \end{align}
\end{enumerate}
where we used \cite[Exercise 61 p. 337]{bor20} stating that $\sum_{p \mid n} 1/p < \max \left( 1, \log \log \log n + 1 \right) $ when $n \geqslant 3$. The first inequality in Corollary~\ref{cor:valuation_sigma_k} follows from inserting \eqref{eq:maj_sum_1} and~\eqref{eq:maj_sum_2} into \eqref{eq:th_1_new}. The second inequality comes from the well known bounds $\varphi(\varpi_1) \leqslant \varpi_1 - 1 \leqslant q-2$, since $\varpi_1 \geqslant 2$, implying that $M_{q,k} \leqslant q-2$. \qedhere

\end{document}